\documentclass[oneside]{amsart}

\usepackage{amssymb}
\usepackage{amsthm}
\usepackage{amscd}
\usepackage{enumitem}

\numberwithin{equation}{section}
\theoremstyle{plain}

\newtheorem{thm}{Theorem}[section]
 
 \newtheorem{lemma}[thm]{Lemma}
\newtheorem{prop}[thm]{Proposition}

\theoremstyle{definition}

\newcommand{\dlabel}[1]{\ifmmode \text{\ttfamily \upshape [#1] } \else
{\ttfamily \upshape [#1] }\fi \label{#1}}

\newcommand{\gen}[1]{\left < #1 \right >}

\begin{document}

\setlength{\baselineskip}{15pt}

\title{A characterization of finite $p$-groups by their Schur multiplier}

\author{Sumana Hatui}
\address{School of Mathematics, Harish-Chandra Research Institute, Chhatnag Road, Jhunsi, Allahabad 211019, INDIA}
\address{\& Homi Bhabha National Institute, Training School Complex, Anushakti Nagar, Mumbai 400085, India}
\email{sumanahatui@hri.res.in, sumana.iitg@gmail.com}


\begin{abstract}
Let $G$ be a finite $p$-group of order $p^n$ and $M(G)$ be its Schur multiplier. It is well known result by Green that $|M(G)|= p^{\frac{1}{2}n(n-1)-t(G)}$ for some $t(G) \geq 0$. In this article we classify non-abelian $p$-groups $G$ of order $p^n$ for $t(G)=\log_p(|G|)+1$.
\end{abstract}

\maketitle

\section{Introduction}
The Schur multiplier $M(G)$ of a group $G$ was introduced by Schur \cite{IS1} in 1904 while studying of projective representation of groups. In 1956, Green \cite{JG} gave an upper bound $p^{\frac{1}{2}n(n-1)}$ on the order of the Schur Multiplier $M(G)$ for $p$-groups $G$ of order $p^n$. So there is an integer $t(G) \geq 0$ such that $|M(G)|=p^{\frac{1}{2}n(n-1)-t(G)}$. This integer $t(G)$ is called corank of $G$ defined in \cite{EW}.
It is an interesting problem to classify the structure of all non-abelian $p$-groups $G$ by the order of the Schur multiplier $M(G)$, i.e., when $t(G)$ is known. Several authors studied this problem for various values of $t(G)$.

First Berkovich \cite{BY} and Zhou \cite{ZH} classified all groups $G$ for $t(G)=0,1,2$. Ellis \cite{EG} also classified groups $G$ for $t(G)=0,1,2,3$ by a different method. After that several authors have classified the groups of order $p^n$ for $t(G)=4,5,6$ in \cite{PN3,PN1, SHJ}.
%We observe in these classifications that number of groups are increasing when the order of Schur multiplier is decreasing. 

Peyman Niroomand \cite{PN} improved the Green's bound and showed that for non abelian $p$-groups of order $p^n$, $|M(G)|=p^{\frac{1}{2}(n-1)(n-2)+1-s(G)}$, for some $s(G) \geq 0$. This integer $s(G)$ is called generalized corank of $G$ defined in \cite{PN5}. The structure of non-abelian $p$-groups for $s(G) = 0,1,2$ had been determined in \cite{PN2,PN4} which is the same as to classify group $G$ for $t(G)=\log_p(|G|)-2, \log_p(|G|)-1,\log_p(|G|)$ respectively.
  
%The list of groups occuring in the classification becomes longer for large value of $s(G)$ and $t(G)$. So it seems to be difficult to go further.
In this paper, we take this line of investigation and classify all non-abelian finite $p$-groups $G$ for which $t(G) = \log_p(|G|)+1$,  which is same as classifing $G$ for $s(G)=3$, i.e., $|M(G)|=p^{\frac{1}{2}n(n-3)-1}$.

Before stating our main result we set some notations.
By $ES_p(p^3)$ we denote the extra-special $p$-group of order $p^3$ having exponent $p$. By $\mathbb{Z}_p^{(k)}$ we denote $\mathbb{Z}_p \times \mathbb{Z}_p \times \cdots \times \mathbb{Z}_p$($k$ times). For a group $G$, $\gamma_i(G)$ denotes the $i$-th term of the lower central series of group $G$ and $G^{ab}$ denotes the quotient group $G/\gamma_2(G)$. We denote $\gamma_2(G)$ by $G'$. 
A group $G$ is called capable group if there exists a group $H$ such that $G \cong H/Z(H)$, where $Z(H)$ denotes center of $H$.
We denote the epicenter of a group $G$ by $Z^*(G)$, which is the smallest central subgroup of $G$ such that $G/Z^*(G)$ is capable.

James \cite{RJ} classified all $p$-groups of order $p^n$ for $n \leq 6$ upto isoclinism which are denoted by $\Phi_k$. We use his notation throughout this paper.

Our main theorem is the following:
\begin{thm}{\bf(Main Theorem)}
Let $G$ be a finite non-abelian $p$-group of order $p^n$ with $t(G)=\log_p(|G|)+1$. Then for odd prime $p$, $G$ is isomorphic to one of the following groups:
\begin{enumerate}
\item $\Phi_2(22)= \langle{\alpha,\alpha_1,\alpha_2 \mid [\alpha_1,\alpha]=\alpha^{p}=\alpha_2, \alpha_1^{p^2}=\alpha_2^p=1\rangle}$,
\item $
\Phi_3(211)a = \langle{\alpha,\alpha_1,\alpha_2, \alpha_3 \mid [\alpha_1,\alpha]=\alpha_2, [\alpha_2,\alpha]=\alpha^p=\alpha_3 \alpha_1^{(p)}=\alpha_2^p=\alpha_3^p=1\rangle},$

\item $\Phi_3(211)b_r = \langle \alpha,\alpha_1,\alpha_2, \alpha_3 \mid [\alpha_1,\alpha]=\alpha_2, [\alpha_2,\alpha]=\alpha^p=\alpha_3,\alpha_1^{(p)}=\alpha_2^p=\alpha_3^p=1 \rangle $ 

\item $\Phi_2(2111)c =\Phi_2(211)c \times \mathbb{Z}_p$, where $\Phi_2(211)c = \langle \alpha,\alpha_1,\alpha_2 \mid [\alpha_1,\alpha]=\alpha_2, \alpha^{p^2}=\alpha_1^p=\alpha_2^p=1 \rangle$,

\item $\Phi_2(2111)d = ES_p(p^3) \times \mathbb{Z}_{p^2}$,

\item $\Phi_3(1^5)=\Phi_3(1^4) \times \mathbb{Z}_p$,
where $\Phi_3(1^4) = \langle \alpha,\alpha_1,\alpha_2,\alpha_3 \mid [\alpha_i,\alpha]=\alpha_{i+1},\alpha^p=\alpha_i^{(p)}=\alpha_3^p=1(i=1,2)\rangle$,

\item $\Phi_7(1^5)=\langle \alpha,\alpha_1,\alpha_2,\alpha_3,\beta \mid [\alpha_i,\alpha]=\alpha_{i+1},[\alpha_1,\beta]=\alpha_3, \alpha^p=\alpha_1^{(p)}=\alpha_{i+1}^p=\beta^p=1 (i=1,2)\rangle$,

\item $\Phi_{11}(1^6)=\langle \alpha_1,\beta_1,\alpha_2,\beta_2,\alpha_3, \beta_3 \mid [\alpha_1,\alpha_2]=\beta_3, [\alpha_2,\alpha_3]=\beta_1, [\alpha_3,\alpha_1]=\beta_2,\alpha_i^{(p)}=\beta_i^p=1(i=1,2,3)\rangle$,

\item $\Phi_{12}(1^6)=ES_p(p^3) \times ES_p(p^3)$,

\item $\Phi_{13}(1^6)=\langle \alpha_1,\alpha_2,\alpha_3,\alpha_4,\beta_1,\beta_2 \mid  [\alpha_i, \alpha_{i+1}]=\beta_i, [\alpha_2,\alpha_4]=\beta_2, \alpha_i^p=\alpha_3^p=\alpha_4^p=\beta_i^p=1(i=1,2)\rangle$,

\item $\Phi_{15}(1^6)=\langle \alpha_1,\alpha_2,\alpha_3,\alpha_4,\beta_1,\beta_2 \mid  [\alpha_i, \alpha_{i+1}]=\beta_i, [\alpha_3,\alpha_4]=\beta_1,[\alpha_2,\alpha_4]=\beta_2^g, \alpha_i^p =\alpha_3^p = \alpha_4^p =\beta_i^p=1(i=1,2)\rangle$,
where $g$ is non-quadratic residue modulo $p$,

\item $(\mathbb{Z}_p^{(4)} \rtimes \mathbb{Z}_p) \times \mathbb{Z}_p^{(2)}$.\\
Moreover for $p=2$, $G$ is isomorphic to one of the following groups:
\item $\mathbb{Z}_2^{(4)} \rtimes \mathbb{Z}_2$,
\item $\mathbb{Z}_2 \times((\mathbb{Z}_4 \times \mathbb{Z}_2) \rtimes \mathbb{Z}_2)$,
\item $\mathbb{Z}_4 \rtimes \mathbb{Z}_4$,
\item $ D_{16}$, the Dihedral group of order $16$.\\
\end{enumerate} 
\end{thm}
\section{Preliminaries}
In this section we list following results which will be used in the proof of our main theorem.
\begin{thm}$($\cite[Theorem 4.1]{MRRR}$)$\label{J}
Let $G$ be a finite group and $K$ a central subgroup of $G$. Set $A = G/K$. Then
$|M(G)||G'\cap K|$ divides $|M(A)| |M(K)| |A^{ab} \otimes K|$.
\end{thm}
%\begin{thm}(see \cite[Corollary 4.16]{BT})\label{SHH}
%Let $G$ be an extra special p-group of order $p^{2n+1}$.\\
%(i) If $n \geq 2$, then $|M(G)| = p^{2n^2-n-1}.$\\
%(ii) If $n = 1$, then the Schur multiplier of $D_8 , Q_8 , ES_p(p^3)$ and $ES_{p^2}(p^3)$ are isomorphic to $\mathbb{Z}_2, 1, \mathbb{Z}_p \times \mathbb{Z}_p$ and $1$ respectively.
%\end{thm}
The following result gives $M(G)$ of non-abelian $p$-groups $G$ of order $p^4$ for $|G'|=p$, follows from \cite{KO} and for $|G'|=p^2$, follows from \cite[page. 4177]{EG} . 
\begin{thm}\label{SHHH}
Let $G$ be a non-abelian $p$-group of order $p^4$, $p$ odd.\\
(i) For $|G'|=p$, $G \cong \Phi_2(211)a,$ $\Phi_2(1^4),$ $\Phi_2(31), \Phi_2(22),$ $\Phi_2(211)b$ or $\Phi_2(211)c$. $M(G) \cong \mathbb{Z}_p \times \mathbb{Z}_p,$  $\mathbb{Z}_p^{(4)}$, ${1}$, $\mathbb{Z}_p,$ $\mathbb{Z}_p \times \mathbb{Z}_p,$ or $\mathbb{Z}_p \times \mathbb{Z}_p$ respectively. \\
(ii) For $|G'|=p^2$, $G \cong \Phi_3(211)a,$ $\Phi_3(211)b_r$ or $\Phi_3(1^4)$. $M(G) \cong \mathbb{Z}_p$, $\mathbb{Z}_p$ or $\mathbb{Z}_p \times \mathbb{Z}_p$ respectively.
\end{thm}
Now we explain a method by Blackburn and Evens \cite{BE} for computing Schur multiplier of $p$-groups $G$ of class $2$ with $G/G'$ elementary abelian.
We can view $G/G'$ and $G'$ as vector spaces over $GF(p)$, which we denote by $V, W$ respectively. Let $v_1,v_2 \in V$ such that $v_i=g_iG', i \in \{1,2\}$ and consider a bilinear mapping $(,)$ of $V$ into $W$ defined by $(v_1,v_2)=[g_1,g_2]$.

Let $X_1$ be the subspace of $V \otimes W$ spanned by all elements of type\\
\centerline{$v_1 \otimes (v_2,v_3) + v_2 \otimes (v_3,v_1) + v_3 \otimes (v_1,v_2), $ for $v_1,v_2,v_3 \in V$.} 

Consider a map $f:V \rightarrow W$ given by $f(gG')=g^p$, $g \in G$. We denote by $X_2$ the subspace spanned by all $v \otimes f(v)$ for $v \in V$. Let $X_1+X_2$ be denoted by $X$, which will be used throughout this paper without further reference.
Now the following result follows from \cite{BE}.
\begin{thm}\label{BEE}
Let $G$ be a $p$-group of class $2$ with $G/G'$ elementary abelian. Then
$|M(G)|/|N|=|V \wedge V|/|W|$, where $N \cong (V \otimes W)/X$.
\end{thm}
Let $G$ be a finite $p$-group of nilpotency class $3$ with centre $Z(G)$. Set $\bar{G}=G/Z(G)$. 
Define a homomorphism $\psi_2 :\bar{G}^{ab} \otimes \bar{G}^{ab} \otimes \bar{G}^{ab}  \rightarrow \frac{\gamma_2(G)}{\gamma_3(G)} \otimes \bar{G}^{ab}$ such that \\
$\psi_2(\bar{x}_1 \otimes \bar{x}_2 \otimes \bar{x}_3)=[x_1,x_2]_{\gamma} \otimes \bar{x}_3 + [x_2,x_3]_{\gamma} \otimes \bar{x}_1 + [x_3,x_1]_{\gamma} \otimes \bar{x}_2$, where  $\bar{x}$ denotes the image in $\bar{G}$ of the element $x \in G$ and $[x,y]_{\gamma}$ denotes the image in $\frac{\gamma_2(G)}{\gamma_3(G)}$ of the commutator $[x,y] \in G$. Define another homomorphism $\psi_3 : \bar{G}^{ab} \otimes \bar{G}^{ab} \otimes \bar{G}^{ab} \otimes \bar{G}^{ab} \rightarrow \gamma_3(G) \otimes \bar{G}^{ab}$ such that
\[\psi_3(\bar{x}_1 \otimes \bar{x}_2 \otimes \bar{x}_3 \otimes \bar{x}_4)=[[x_1,x_2],x_3] \otimes \bar{x}_4 + [x_4,[x_1,x_2]] \otimes \bar{x}_3 + [[x_3,x_4],x_1]\otimes \bar{x}_2 +[x_2,[x_3,x_4]] \otimes \bar{x}_1.\]
\begin{thm}$($\cite[Proposition 1]{EW} and \cite{E}$)$\label{RM}
Let $G$ be a finite $p$-group of nilpotency class $3$. With the notations above, we have \\
\centerline{$|M(G)||\gamma_2 (G)||Image(\psi_2)||Image(\psi_3)| \leq |M(G^{ab})||\frac{\gamma_2(G)}{\gamma_3(G)} \otimes \bar{G}^{ab}||\gamma_3(G)\otimes \bar{G}^{ab}|$.}
% \cdots |\gamma_c G \otimes \bar{G}^{ab}|$
\end{thm}
\section{Proof of Main Theorem}
In this section we prove our main theorem. Proof is divided into several parts depending on the structure of the group. We start with the following lemma which establishes the result for groups of order $p^n$ for $n \leq 5$.
\begin{lemma}
Let $G$ be a non-abelian $p$-group of order $p^n$ $(n \leq 5)$ with $t(G)=\log_p(|G|)+1$. Then for odd prime $p$,
$G \cong \Phi_2(22), \Phi_3(211)a, \Phi_3(211)b_r, \Phi_2(2111)c,$ $\Phi_2(2111)d, \Phi_3(1^5)$ or $\Phi_7(1^5)$, and for $p=2$,
$G \cong \mathbb{Z}_2^{(4)} \rtimes \mathbb{Z}_2, \mathbb{Z}_2 \times((\mathbb{Z}_4 \times \mathbb{Z}_2) \rtimes \mathbb{Z}_2), D_{16}$ or $\mathbb{Z}_4 \rtimes \mathbb{Z}_4$.
\end{lemma}
\begin{proof}
For groups of order $p^4$ the result follows from Theorem \ref{SHHH}. For groups of order $p^5$ result follows from  \cite{SHJ} and \cite[Main Theorem]{SIX}. For $p=2$, the result follows from computation in HAP \cite{HAP} package of GAP \cite {GAP}.             
\hfill$\Box$

\end{proof}
The following lemma easily follows from \cite[Main Theorem]{PN}. 
\begin{lemma}
There is no non-abelian $p$-group $G$ with $|G'|\geq p^4$ and $t(G) =\log_p(|G|)+1$.
\end{lemma}
\begin{lemma}\label{EA}
Let $G$ be a non-abelian $p$-group of order $p^n$ $(n \geq 6)$ with $t(G) \leq \log_p(|G|)+1$. Then $G^{ab}$ is an elementary abelian $p$-group.
\end{lemma}
\begin{proof}
Let $|G'|=p^k$. Suppose that $G^{ab}$ is not elementary abelian and $\bar{G}:=G/Z(G)$ is a $\delta$-generator group. Then $\delta \leq (n-k-1)$ and $|M(G^{ab})| \leq p^{\frac{1}{2}(n-k-1)(n-k-2)}$ by \cite[Lemma 2.2]{PN}.
Note that $|\frac{\gamma_2(G)}{\gamma_3 (G)} \otimes \bar{G}^{ab}||\frac{\gamma_3 (G)}{\gamma_4 (G)} \otimes \bar{G}^{ab}| \cdots |\gamma_c (G)\otimes \bar{G}^{ab}|=|(\frac{\gamma_2 (G)}{\gamma_3 (G)} \oplus \frac{\gamma_3 (G)}{\gamma_4 (G)} \oplus \cdots \gamma_c (G)) \otimes\bar{G}^{ab} | \leq p^{k\delta}$ and $|$Image$(\psi_2)| \geq p^{\delta-2}$.
Now it follows from \cite[Proposition 1]{EW} that \\
\centerline{$|M(G)| \leq p^{\frac{1}{2}(n-k-1)(n-k-2)+(k-1)\delta-(k-2)}$,}
which gives $|M(G)| \leq p^{\frac{1}{2}n(n-3)-\frac{1}{2}(k^2-k)-n+4}$, a contradiction for $n \geq 6$.
\hfill$\Box$

\end{proof}
\begin{lemma}\label{EXP}
Let $G$ be a non-abelian $p$-group of order $p^n$ $(n \geq 6)$ and $|G'|=p, p^2$ or $p^3$ with $t(G)\leq \log_p(|G|)+1$. Then $Z(G)$ is of exponent at most $p^2,p$ or $p$ respectively. 
\end{lemma}
\begin{proof}
Suppose that $|G'|=p$. Let the exponent of $Z(G)$ be $p^k$ $(k \geq 3)$ and $K$ be a cyclic central subgroup  of order $p^k$. Then using Theorem \ref{J}, we have\\
\centerline{$|M(G)| \leq p^{-1}|M(G/K)||{(G/K)}^{ab} \otimes K| \leq p^{-1}p^{\frac{1}{2}(n-k)(n-k-1)}p^{(n-k)} \leq p^{\frac{1}{2}(n-1)(n-4)}$,} \\
which gives a contradiction.
Similarly we can prove the result for $|G'|=p^2, p^3$.
\hfill$\Box$

\end{proof}
First we consider the groups $G$ such that $|G'|=p$.
\begin{lemma}
There is no non-abelian $p$-group $G$ of order $p^n$ with $|G'|=p$ and $t(G) = \log_p(|G|)+1$.
\end{lemma}
\begin{proof}
Note that $G'$ is central subgroup of $G$. By \cite[Theorem 3.1]{MRRR}, we have $|M(G)||G'| \geq |M(G/G')|$ and by Lemma \ref{EA} $|M(G/G')|=p^{\frac{1}{2}(n-1)(n-2)}$. Therefore $|M(G)|\geq p^{\frac{1}{2}(n-1)(n-2)-1}$, which is a contradiction.
\hfill$\Box$

\end{proof}
Now we consider groups such that $|G'|=p^2$.
\begin{lemma}\label{BH}
Let $G$ be a $p$-group of order $p^n$ $(n \geq 6)$ with $|G'|=p^2$ and $t(G) \leq \log_p(|G|)+1$. If $K$ is a cyclic central subgroup of order $p$ in $G' \cap Z(G)$, then $G/K$ is isomorphic to one of the following groups:
$ES(p^3) \times \mathbb{Z}_p^{(n-4)}, E(2) \times \mathbb{Z}_p^{(n-2m-3)},$ $ ES(p^{2m+1}) \times \mathbb{Z}_p^{(n-2m-2)}, D_8 \times \mathbb{Z}_2^{(n-4)}, Q_8 \times \mathbb{Z}_2^{(n-4)}$, where $ES(p^3),$ $ES(p^{2m+1})$ denotes extra special $p$-group of order $p^3$ and $p^{2m+1}$ $(m \geq 2)$ respectively, and $E(2)$ denotes central product of an extra special $p$-group of order $p^{2m+1}$ and a cyclic group of order $p^2$.
\end{lemma}
\begin{proof}
Suppose that $G$ is a $p$-group with $|G'|=p^2$. Now consider a cyclic central subgroup $K$ of order $p$ in $G' \cap Z(G)$. Then by Theorem \ref{J} and \cite[Main Theorem]{PN} we have 

\centerline{$|M(G)| \leq |M(G/K)|p^{(n-3)} \leq p^{\frac{1}{2}(n-2)(n-3)+1}p^{n-3}=p^{\frac{1}{2}n(n-3)+1}$}

Now using \cite[Theorem 21 and Corollary 23]{PN2} and \cite[Theorem 11]{PN4} for group $G/K$, we get our result.
\hfill$\Box$

\end{proof}
\begin{prop}
There is no non-abelian $p$-group $G$ of order $p^n$ $(n \geq 6)$ with $|G'|=p^2, |Z(G)|=p$ and $t(G) \leq \log_p(|G|)+1$.
\end{prop}
\begin{proof}
%Using Theorem \ref{J} we get $|M(G)|p \leq |M(G/Z(G))|p^{(n-2)}$ \hspace{1.5cm}$\cdots$  (1)
%Now by \cite[Theorem 23, Theorem 21]{PN2} we have,  
%$|M(G/Z(G))|=p^{\frac{1}{2}(n-2)(n-3)+1}$ if and only if $G/Z(G) \cong ES_p(p^3) %\times \mathbb{Z}_p^{(n-4)}$ and
%$|M(G/Z(G))|=p^{\frac{1}{2}(n-2)(n-3)}$ if and only if $G/Z(G) \cong D_8 \times \mathbb{Z}_2^{(n-4)}$. By \cite[Theorem 11]{PN4} we have
%$|M(G/Z(G))|=p^{\frac{1}{2}(n-2)(n-3)-1}$ if and only if $G/Z(G) \cong E(2) \times \mathbb{Z}_p^{(n-2m-3)},ES_{p^2}(p^3) \times \mathbb{Z}_p^{(n-4)}, Q_8 \times \mathbb{Z}_2^{(n-4)}, H \times \mathbb{Z}_p^{(n-2m-2)}$ , where $H$ is extra special $p$-group of order $p^{(2m+1)}(m \geq 2)$.  
%If $G/Z(G)$ is not isomorphic to the above groups, then by (1) we get $|M(G)| \leq p^{\frac{1}{2}n(n-3)-2}$ which is not our case. So $G/Z(G)$ must be of the above forms. Observe that for all these structures of $G/Z(G)$,
%Let $Z(G)=<\gamma>$ and $G' = <z> \times <\gamma>$.\\
By Lemma \ref{EA}, $G/G'$ is elementary abelian group of order $p^{n-2}$ . 
%Observe that $G/Z(G)$ is isomorphic to one of the groups given in Theorem \ref{BH}.
Thus $G$ is $(n-2)$-generator group. We can choose generators $x,y,\beta_1, \beta_2$ $ \cdots ,\beta_{n-4}$  of $G$ such that $[x,y]=z \notin Z(G)$. Now we claim that $|$Image$(\psi_3)| \geq p^{n-4}$.

If $[z,x]$ is non-trivial in $Z(G)$, then $\psi_3(x \otimes y \otimes x \otimes \beta_i)$ $i=1, \cdots ,n-4$ gives $n-4$ linearly independent elements of $\gamma_3(G) \otimes \bar{G}^{ab}$.
By symmetry if $[z,y]$ is non-trivial in $Z(G)$, then we have a similar conclusion.
On the other hand if both are trivial i.e., $[z,x]=[z,y]=1$, then $[z,\beta_k]$ is non-trivial in $Z(G)$ for some $\beta_k$ and $\psi_3(x \otimes y \otimes \beta_k \otimes \beta_i)(i \neq k)$ give $n-5$ linearly independent elements of $\gamma_3(G) \otimes \bar{G}^{ab}$. Hence $|$Image$(\psi_3)| \geq p^{n-5}$. Note that $|$Image$(\psi_2)| \geq p^{n-4}$. So by Theorem \ref{RM} we have

\centerline{$p^2|M(G)||$Image$(\psi_2)||$Image$(\psi_3)|  \leq p^{\frac{1}{2}(n-2)(n-3)}p^{2(n-2)}$.} 
It follows that $|M(G)| \leq p^{\frac{1}{2}n(n-3)-n+6}$, which is a contradiction for $n \geq 8$.

Now if either $\psi_3(x \otimes y \otimes \beta_k \otimes x)$ or  $\psi_3(x \otimes y \otimes y \otimes \beta_k)$ is non-trivial then $|$Image$(\psi_3)| \geq p^{n-4}$ and

\centerline{$p^2|M(G)||$Image$(\psi_2)||$Image$(\psi_3)|  \leq p^{\frac{1}{2}(n-2)(n-3)}p^{2(n-2)}$.} 
It follows that $|M(G)| \leq p^{\frac{1}{2}n(n-3)-n+5}$, which is a contradiction for $n \geq 7$. 
Otherwise suppose $\psi_3(x \otimes y \otimes \beta_k \otimes x)=\psi_3(x \otimes y \otimes y \otimes \beta_k)=1$, then $[x,y,\beta_k]=[\beta_k,x,y]=[y,\beta_k,x]$  and $p=3$. By HAP \cite{HAP} of GAP \cite{GAP} there is no group $G$ of order $3^7$ with $|G'|=3^2, |Z(G)|=3$ and $|M(G)|=3^{13}$.

For $|G|=p^6$ $(p \neq 2)$, by \cite{RJ} it follows that $G$ belongs to the isoclinism class $\Phi_{22}$. In this case $|$Image$(\psi_2)| \geq p^2$ and $|$Image$(\psi_3)| \geq p^3$. Hence it follows from Theorem \ref{RM} that $|M(G)| \leq p^7$, which is not our case.

For $p=2$, there is no group $G$ of order $2^6$ which satisfies the given hypothesis, follows from computation with HAP \cite{HAP} of GAP \cite{GAP}.
\hfill$\Box$

\end{proof}
\begin{lemma}\label{T}
Let $G$ be a non-abelian $p$-group of order $p^n$ $(n \geq 6)$ with $t(G)=\log_p(|G|)+1$ and $|G'|=p^2$. If there exists a central subgroup $K$ of order $p$ such that $K \cap G'=1$, then $G/K$ is isomorphic to either $\mathbb{Z}_p^{(4)} \rtimes \mathbb{Z}_p$ or $(\mathbb{Z}_p^{(4)} \rtimes \mathbb{Z}_p) \times \mathbb{Z}_p$ and $p$ is odd.
\end{lemma}
\begin{proof}
By Theorem \ref{J} and \cite[Main Theorem]{PN}, we have 

\centerline{$|M(G)| \leq |M(G/K)|p^{(n-3)} \leq p^{\frac{1}{2}(n-1)(n-4)+1+n-3}=p^{\frac{1}{2}n(n-3)}$.} 

Note that $|G/K| \geq p^5$ with $(G/K)'=p^2$.
Now we have $|M(G/K)|=p^{\frac{1}{2}(n-1)(n-4)+1}$ if and only if $G/K \cong \mathbb{Z}_p^{(4)} \rtimes \mathbb{Z}_p$ $(p \neq 2)$ by \cite[Theorem 21]{PN2} and  $|M(G/K)|=p^{\frac{1}{2}(n-1)(n-4)}$ if and only if $G/K \cong (\mathbb{Z}_p^{(4)} \rtimes \mathbb{Z}_p) \times \mathbb{Z}_p$ $(p \neq 2)$ by \cite[Theorem 11]{PN4}.
\hfill$\Box$

\end{proof}
\begin{prop}\label{K}
There is no non-abelian $p$-group of order $p^7$ with $G' =Z(G) \cong \mathbb{Z}_p \times \mathbb{Z}_p$ and $t(G) = \log_p(|G|)+1$.
\end{prop}
\begin{proof}
Note that if $Z^*(G)$ contains any central subgroup of $G$, then $|M(G)| < p^{13}$, by \cite[Theorem 2.5.10]{GK}, which is not our case. So we have $Z^*(G)=1$, i.e, $G$ is capable.

First we consider groups $G$ of order $p^7$ of exponent $p^2$, for odd $p$. Note that $G/G'$ is elementary abelian of order $p^5$ by Lemma \ref{EA}. So we can take a generating set $\{ \beta_1,\beta_2,\beta_3,\beta_4,\beta_5 \}$ of $G$ and $\{\eta, \gamma\}$ of $G'$. Here either $|G^p|=p$ or $G^p=G'$. We claim that $|X| \geq p^8$.

Let $|G^p|=p$. Without loss of generality assume that $\eta$ is $p$-th power of some $\beta_k$, say $\beta_1$, $[\beta_i,\beta_j] \notin \gen{\eta}$ and all $\beta_k$'s $(k > 1)$ are of order $p$. Then $\langle \beta_i \otimes \beta_1^p, i \in \{1,2,3,4,5\}\rangle$ is a subspace of $X_2$ and $\langle \psi_2(\beta_k \otimes \beta_i \otimes \beta_j), k \in \{1,2,3,4,5\}, k \neq i,j \rangle$ is a subspace of $X_1$. For $G^p=G'$, without loss of generality, assume that both $\eta, \gamma$ are $p$-th power of some $\beta_{k_1},\beta_{k_2}$, say $\beta_1$ and $\beta_2$ respectively and all other $\beta_i$'s are of order $p$, then $\langle \beta_i \otimes \beta_1^p, \beta_j \otimes \beta_2^p, i \in \{1,3,4,5\}, j \in \{2,3,4,5\}\rangle$ is a subspace of $X_2$.

Hence we observe that for non-abelian group $G$ of order $p^7$ of exponent $p^2$, $|X| \geq p^8$ and by Theorem \ref{BEE}, $|M(G)| < p^{13}$, a contradiction. 

Now consider groups of order $p^7$ of exponent $p$. By \cite{HKM}  it follows that there is only one capable group 
\[ G =\gen{x_1,\cdots , x_5, c_1, c_2 \mid [x_2, x_1] =[x_5, x_3]= c_1, [x_3, x_1] = [x_5, x_4] = c_2}\] upto isomorphism.
By Theorem \ref{BEE} we have $|M(G)|=p^9$ as $|X|=|X_1|=p^9$ which is not our case.
\hfill$\Box$

\end{proof}
The following result weaves the next thread in the proof of Main Theorem.
\begin{thm}
Let $G$ be a non-abelian $p$-group of order $p^n$ $(n \geq 6)$ with $|G'|=p^2$,$ |Z(G)| \geq p^2$ and $t(G)=\log_p(|G|)+1$. Then $G$ is isomorphic to $\Phi_{12}(1^6),\Phi_{13}(1^6),$ $\Phi_{15}(1^6)$ or $(\mathbb{Z}_p^{(4)} \rtimes \mathbb{Z}_p) \times \mathbb{Z}_p^{(2)}$. Moreover, $p$ is always odd.
\end{thm}
\begin{proof}
By Lemma \ref{EXP}, $Z(G)$ is of exponent $p$. We consider two cases here.\\
{\bf Case 1}:
Let $G'=Z(G) \cong \mathbb{Z}_p \times  \mathbb{Z}_p \cong \gen{z} \times \gen{w}$. 
The isomorphism type of $G/K$ is as in Theorem \ref{BH}.
It follows from these structures that there are $n-2$ generators $\{x,y,\alpha_i, i \in \{1,2,\cdots n-4\}\}$ of $G$ such that $[x,y]\in \gen{z}$ and $[\alpha_k,x] \in \gen{w}$, for some $k$. Hence $\psi_2(x \otimes y \otimes \alpha_i, i \in \{1,2,\cdots n-4\})$ and $\psi_2(\alpha_k \otimes x \otimes \alpha_i, i \in \{1,2,\cdots n-4\}, i \neq k)$ gives $(2n-9)$ linearly independent elements in $G' \otimes G/G'$.
Now by Theorem \ref{BEE}, we have $|M(G)| \leq p^{\frac{1}{2}n(n-3)-n+6}$, which is possible only when $n \leq 7$. Now it only remains to consider groups of order $p^6$ and $p^7$.

For groups of order $p^7$ result follows from Proposition \ref{K}.

Now consider groups of order $p^6$ for odd $p$. Then $G$ belongs to the isoclinism classes $\Phi_{12}, \Phi_{13}$ or $\Phi_{15}$ of \cite{RJ}. If $G$ is of exponent $p^2$, then it is easy to see that $|X| \geq p^5$ and $|M(G)| \leq p^7$ by Theorem \ref{BEE}.
Now for $\Phi_{12}(1^6),\Phi_{13}(1^6),\Phi_{15}(1^6)$ we have $|X|=|X_1|=p^4$ and using Theorem \ref{BEE} we see that all of theses groups have Schur multiplier of order $p^8$.

By HAP \cite{HAP} of GAP \cite{GAP} we see that there is no such group for $p=2$. \\
{ \bf Case 2}: Consider complement of case 1. In these cases the hypothesis of Lemma \ref{T} holds. We can choose a central subgroup $K$ of order $p$ such that $K \cap G'=1$ and $G/K \cong \mathbb{Z}_p^{(4)} \rtimes \mathbb{Z}_p$ or $(\mathbb{Z}_p^{(4)} \rtimes \mathbb{Z}_p) \times \mathbb{Z}_p$. Here $|Z(G)/K|=|Z(G/K)|$. 
Also note that $G$ is of exponent $p$. Then it follows easily that $G \cong (\mathbb{Z}_p^{(4)} \rtimes \mathbb{Z}_p) \times \mathbb{Z}_p^{(2)}$ $(p \neq 2)$.

\hfill$\Box$

\end{proof}
Finally we consider groups $G$ such that $|G'|=p^3$.
\begin{lemma}\label{RM1}
Let $G$ be a non-abelian $p$-group of order $p^n$ with $t(G)=\log_p(|G|)+1$ and $|G'|=p^3$. Then for any subgroup $K \subseteq Z(G) \cap G'$ of order $p$, $G/K \cong \mathbb{Z}_p^{(4)} \rtimes \mathbb{Z}_p$ $(p \neq 2)$. In particular, $|G|=p^6$.\\
\end{lemma}
\begin{proof}
For $p$ odd, by Theorem \ref{J} we have $|M(G)| p \leq |M(G/K)||G/G' \otimes K|$. Since $G/G'$ is elementary abelian by Lemma \ref{EA}, we have $|M(G)| \leq |M(G/K)|p^{(n-4)}$ and by \cite[Main Theorem]{PN}, $|M(G/K)| \leq p^{\frac{1}{2}(n-1)(n-4)+1}$. Hence $|M(G)| \leq p^{\frac{1}{2}n(n-3)-1}$. Using \cite[Theorem 21]{PN2}, we get $G/K \cong \mathbb{Z}_p^{(4)} \rtimes \mathbb{Z}_p$.
For $p=2$, $|M(G)| < p^{\frac{1}{2}n(n-3)-1}$, which is not our case.
\hfill$\Box$

\end{proof}
\begin{lemma}
There is no non-abelian $p$-group $G$ with $|G'|=p^3, |Z(G)|=p$ and $t(G)=\log_p(|G|)+1$.
\end{lemma}
\begin{proof}
By the  preceeding Lemma we have $G/Z(G) \cong \mathbb{Z}_p^{(4)} \rtimes \mathbb{Z}_p \cong \Phi_4(1^5)$ and $|G|=p^6$. From \cite{RJ} we can see that $G$ belongs to one of the isoclinism classes $\Phi_{31}, \Phi_{32}, \Phi_{33}$. 
Observe that for these groups $|$Image$(\psi_2)|\geq p$ and $|$Image$(\psi_3)|\geq p$.
Now using Theorem \ref{RM} we get $|M(G)| \leq p^7$, which is not our case.
\hfill$\Box$

\end{proof}
The following theorem now completes the proof of the Main Theorem.
\begin{thm}
Let $G$ be a non-abelian $p$-group of order $p^n$ with $|G'|=p^3$ and $t(G)=\log_p(|G|)+1$. Then $G \cong \Phi_{11}(1^6)$.
\end{thm}
\begin{proof}
We claim that $Z(G) \subseteq G'$.
Assume that $Z(G) \nsubseteq G'$. Then there is a central subgroup $K$ of order $p$ such that $G' \cap K=1$. Now by Theorem \ref{J} and \cite[Main Theorem]{PN} we have

 \centerline{$|M(G)| \leq |M(G/K)||G/G'K| \leq p^{\frac{1}{2}n(n-5)+1+(n-4)}=p^{\frac{1}{2}n(n-3)-3}$,}
which is a contradiction. Hence $Z(G) \subseteq G'$.

Note that $|Z(G)| \geq p^2$ by preceeding lemma. We can now choose two distinct central subgroups $K_i$ $(i=1,2)$ of order $p$. Then by Lemma \ref{RM1} we have $G/K_i \cong \mathbb{Z}_p^{(4)} \rtimes \mathbb{Z}_p$ $(i=1,2)$, which are of exponent $p$. So $G$ is of exponent $p$. Hence from \cite{RJ} it follows that $G \cong \Phi_6(1^6), \Phi_9(1^6), \Phi_{10}(1^6), \Phi_{11}(1^6), \Phi_{16}(1^6), \Phi_{17}(1^6), \Phi_{18}(1^6),$ $ \Phi_{19}(1^6),\Phi_{20}(1^6), \Phi_{21}(1^6)$ groups are of order $p^6$ of exponent $p$ with $|Z(G)| \geq p^2, |G'|=p^3$.

First consider groups $G \cong \Phi_6(1^6), \Phi_9(1^6), \Phi_{10}(1^6)$.  Then by a routine check we can show that $|M(G)| \leq p^6$ using \cite[Theorem 2.2.10]{GK}.

Now consider the group $G =\Phi_{11}(1^6)$. Then $G$ is of class two with $G/G'$ elementary abelian. Hence by Theorem \ref{BEE} it follows that $\Phi_{11}(1^6)$ has Schur multiplier of order $p^8$ as $|X|=|X_1|=p$.

For other groups $G$, observe that $|$Image$(\psi_2)|\geq p$, $|$Image$(\psi_3)|\geq p$ and hence it follows from Theorem \ref{RM} that $|M(G)| \leq p^7$.
\hfill$\Box$

\end{proof}
{\bf Acknowledgement}:
I am grateful to my supervisor Manoj K. Yadav for his guidance, motivation and discussions. I wish to thank the Harish-Chandra Research Institute, the Dept. of
Atomic Energy, Govt. of India, for providing excellent research facility.

\end{document}